\documentclass[11pt]{amsart}
\usepackage{amsmath,amssymb, amsthm,delarray}
\usepackage{graphicx}
\numberwithin{equation}{section}
\newtheorem{thm}{Theorem}
\newtheorem{lemma}[thm]{Lemma}
\newtheorem{remark}[thm]{Remark}
\newtheorem{corollary}[thm]{Corollary}
\newtheorem{proposition}[thm]{Proposition}
{\rm}
\newtheorem{example}{Example}{\rm}
{\rm}

\def\talpha{\tilde{\alpha}}
\def\beq{\begin{equation} }
\def\eeq{\end{equation} }
\def\e{\hbox{\rm e}}

\def\C{\mathbb{C}}

\def\N{\mathbb{N}}
\def\M{\mathbf{M}}
\def\R{\mathbb{R}}
\def\A{\mathbb{A}}
\def\P{\mathbf{P}}
\def\si{\mathbf{\Sigma}}
\def\K{\mathbf{K}}
\def\B{\mathbf{B}}
\def\Q{\mathbf{Q}}
\def\M{\mathbf{M}}

\def\G{\mathcal{G}}

\def\v{\mathbf{v}}
\def\f{\mathbf{f}}
\def\h{\mathbf{h}}

\def\u{\mathbf{u}}
\def\f{\mathbf{f}}

\def\x{\mathbf{x}}

\def\y{\mathbf{y}}
\def\z{\mathbf{z}}
\def\s{\mathcal{S}}

\def\dis{\displaystyle}

\def\supp{{\rm supp}\,}

\def\R{\mathbb{R}}
\def\B{\mathbf{B}}
\def\C{\mathbf{C}}

\def\g{\mathbf{g}}
\def\blambda{\mathbf{\lambda}}

\def\det{{\rm det}\,}
\def\vol{{\rm vol}\,}
\def\dis{\displaystyle}
\def\tv{\tilde{\v}}

\begin{document}
\title[level set]{Recovering an homogeneous polynomial
from moments of its level set}
\author{Jean B. Lasserre}

\begin{abstract}
Let $\K:=\{\x: g(\x)\leq 1\}$ be the compact sub-level set of some homogeneous polynomial
$g$. Assume that the only knowledge about $\K$ is the degree of $g$ as well as 
the moments of the Lebesgue measure on $\K$ up to order $2d$. Then 
the vector of coefficients of $g$ is solution of a simple linear system whose associated matrix is nonsingular.
In other words, the moments up to order $2d$
of the Lebesgue measure on $\K$ encode all information on the homogeneous polynomial $g$ 
that defines $\K$ (in fact, only moments of order $d$ and $2d$ are needed). 

\end{abstract}


\maketitle


\section{Introduction}
The inverse problem of reconstructing a geometrical object $\K\subset\R^n$
from the only knowledge of moments of some measure $\mu$ whose support is $\K$
 is a fundamental problem 
 in both applied and pure mathematics with important applications in e.g.
 computer tomography, inverse potentials, signal processing, and statistics and probability, to cite a few.
In computer tomography, for instance, the X-ray images of an object can be used to estimate the moments of the
underlying mass distribution, from which one seeks to recover the shape of the object that appears on some given
images. In gravimetry applications, the measurements of the gravitational field can be converted into information
concerning the moments, from which one seeks to recover the shape of the source of the anomaly.

Of course, {\it exact} reconstruction of objects $\K\subset\R^n$ is in general impossible unless
$\K$ has very specific properties. For instance, if $\K$ is a convex polytope then
exact recovery of all its vertices has been shown to be possible via a variant of what is known as {\it prony} method. 
Only a rough bound on the number of vertices is required and relatively few moments suffice
for exact recovery. For more details the interested reader is referred to the recent contribution
of Gravin et al. \cite{gravin} and the references therein. On the other hand, 
Cuyt et al. \cite{cuyt} have shown that
{\it approximate} recovery of a general $n$-dimensional shape is possible by using an
 interesting property of multi-dimensional Pad\'e
approximants, analogous to the Fourier slice theorem for the Radon transform.

\subsection*{Contribution} From previous contributions and their references, it is transparent that exact recovery of an $n$-dimensional shape
is a difficult problem that can be solved only in a few cases. And so identifying such cases is of theoretical and practical interest.
The goal of this paper is to identify one such case as we show that
exact recovery is possible when $\K\subset\R^n$ is the (compact) sublevel set $\{\x\in\R^n \,:\,g(\x)\leq 1\}$
associated with an homogeneous polynomial $g$. By exact recovery we mean recovery of {\it all} coefficients of the polynomial $g$.
In fact, exact recovery is not only possible but rather straightforward as it suffices to solve a linear system with
a nonsingular matrix! Moreover, only moments of order $d$ and $2d$ of the Lebesgue measure on $\K$ are needed.
As already mentioned, exact recovery is possible only if $\K$
has very specific properties and indeed, crucial in the proof is a property of levels sets associated with homogeneous polynomials
(and in fact, also true for level sets of positively homogeneous nonnegative functions).

\section{Main result}

\subsection{Notation and definitions}

Let $\R[\x]$ be the ring of polynomials in the variables $\x=(x_1,\ldots,x_n)$ and let
$\R[\x]_d$ be the vector space of polynomials of degree at most $d$
(whose dimension is $s(d):={n+d\choose n}$).
For every $d\in\N$, let  $\N^n_d:=\{\alpha\in\N^n:\vert\alpha\vert \,(=\sum_i\alpha_i)=d\}$, 
and
let $\v_d(\x)=(\x^\alpha)$, $\alpha\in\N^n$, be the vector of monomials of the canonical basis 
$(\x^\alpha)$ of $\R[\x]_{d}$. 
Denote by  $\s_k$ the space of $k\times k$ real symmetric matrices with
scalar product $\langle \B,\C\rangle={\rm trace}\,(\B\C)$; also, the notation $\B\succeq0$ 
(resp. $\B\succ0$) stands for $\B$ is positive semidefinite (resp. positive definite).

A polynomial $f\in\R[\x]_d$ is written
\[\x\mapsto f(\x)\,=\,\sum_{\alpha\in\N^n}f_\alpha\,\x^\alpha,\]
for some vector of coefficients $\f=(f_\alpha)\in\R^{s(d)}$.

A real-valued polynomial $g:\R^n\to\R$ is homogeneous of degree $d$ ($d\in\N$)
if $g(\lambda\x)=\lambda^dg(\x)$ for all $\lambda$ and all $\x\in\R$. 
Given $g\in\R[\x]$, denote by $G\subset\R^n$ the sublevel set $\{\x\,:\,g(\x)\leq 1\}$.

If $g$ is homogeneous then $G$ is compact only if $g$ is nonnegative on $\R^n$
(and so $d$ is even).
Indeed suppose that $g(\x_0)<0$ for some $\x_0\in\R^n$; then by homogeneity,
$g(\lambda \x_0)<0$ for all $\lambda>0$ and so $G$ contains a half-line and cannot be compact.

\subsection{Main result}

The main result is based on the following result of independent interest valid for 
positively homogeneous functions (and not only homogeneous polynomials).
A function $f:\R^n\to\R$ is positively homogeneous of degree $d\in\R$ if
$f(\lambda\x)=\lambda^df(\x)$ for all $\lambda>0$ and all $\x\in\R^n$.
\begin{lemma}
Let $f:\R^n\to\R$ be a measurable, positively homogeneous and nonnegative function of degree $0<d\in\R$, with bounded level set
$\{\x\,:\,f(\x)\leq 1\}$. Then for every $k\in\N$ and $\alpha\in\N^n$:

\begin{equation}
\label{lem1-1}
\int_{\{\x\,:\,f(\x)\leq 1\}}\x^\alpha\,f(\x)^k\,d\x\,=\,\frac{n+\vert\alpha\vert}{n+kd+\vert\alpha\vert}\,\int_{\{\x\,:\,f(\x)\leq 1\}}\,\x^\alpha\,d\x.\\
\end{equation}
\end{lemma}
\begin{proof}
To prove (\ref{lem1-1}) we use an argument already used in Morosov and Shakirov \cite{morosov1,morosov2}.
With $\alpha\in\N^n$, let $\talpha:=(\alpha_2,\ldots,\alpha_n)\in\N^{n-1}$ and define $\z:=(z_2,\ldots,z_n)$.

Let $\phi:\R_+\to\R$ be measurable and  consider the integral $\int_{\R^n}\phi(g(\x))\,\x^\alpha d\x$.
Using the change of variable $x_1=t$ and $x_i=tz_i$ for all $i=2,\ldots,n$, and 
invoking homogeneity, one obtains:
\begin{eqnarray*}
\int_{\R^n}\phi(f(\x))\,\x^\alpha\,d\x&=&\int_{\R^n}\phi(t^df(1,z_2,\ldots,z_n))\,t^{n+\vert\alpha\vert-1}\z^{\talpha}\,d(t,\z)\\
&=&d^{-1}\left(\int_0^\infty u^{(n+\vert\alpha\vert)/d-1}\phi(u)\,du\right)\times A_\alpha\\
\mbox{with $A_\alpha$}&=&\int_{\R^{n-1}}\z^{-\talpha}f(1,\z)^{-(n+\vert\alpha\vert)/d}\,d\z.
\end{eqnarray*}
Hence the choices $t\mapsto \phi(t):={\rm I}_{[0,1]}(t)$ and $t\mapsto \phi(t):=t^k {\rm I}_{[0,1]}(t)$ yield
\begin{eqnarray*}
d\int_{\{\x\,:\,f(\x)\leq 1\}}\x^\alpha d\x&=&A_\alpha\int_0^1u^{(n+\vert\alpha\vert)/d-1}\,du=\frac{A_\alpha d}{n+\vert\alpha\vert}\\
d\int_{\{\x\,:\,f(\x)\leq 1\}}f(\x)^k\,\x^\alpha d\x&=&A_\alpha\int_0^1u^{(n+kd+\vert\alpha\vert)/d-1}\,du=\frac{A_\alpha d}{n+kd+\vert\alpha\vert},
\end{eqnarray*}
respectively. And so (\ref{lem1-1}) follows.
\end{proof}

With $g\in\R[\x]_d$ being an homogeneous polynomial of degree $d$, consider now the matrix $\M_d(\lambda)$ of moments of order $2d$
associated with the Lebesgue measure on $G=\{\x\,:\,g(\x)\leq 1\}$,
that is, $\M_d(\lambda)$ is a real square matrix with rows and columns indexed by the monomials $\x^\alpha$, $\alpha\in\N^n_d$,
and with entries
\begin{equation}
\label{moment-2d}
\M_d(\blambda)[\alpha,\beta]\,=\,\int_{\{\x\,:\,g(\x)\leq 1\}}\x^{\alpha+\beta}\,d\x\,=:\,\lambda_{\alpha+\beta},\qquad\forall\,\alpha,\beta\in\N^n_d.
\end{equation}

So all entries of $\M_d(\lambda)$ are moments of order $2d$.
Our main result is as follows:
\begin{thm}
\label{thmain}
Let $g\in\R[\x]_{d}$ be homogeneous of degree $d$ with unknown coefficient vector
$\g\in\R^{s(d)}$ and with compact level set $G=\{\x\,:\,g(\x)\leq 1\}$. Assume that one knows the moments $\blambda=(\lambda_\alpha)$ for the Lebesgue measure on $G$, for every $\alpha\in\N^n$ with
$\vert\alpha\vert=2d$ and $\vert\alpha\vert=d$. Then:
\begin{equation}
\label{thmain-1}
\g\,=\,\frac{n+d}{n+2d}\:\M_d(\blambda)^{-1}\,\blambda^{(d)}
\end{equation}
where $\blambda^{(d)}=(\lambda_\alpha)$, $\alpha\in\N^n_d$, is the vector of all moments of order $d$.
\end{thm}
\begin{proof}
Use (\ref{lem1-1}) with $k=1$ and $\vert\alpha\vert=d$ to obtain
\[\sum_{\beta\in\N^n:\,\vert\beta\vert=2d}g_\beta\,\lambda_{\alpha+\beta}\,=\,\frac{n+d}{n+2d}\,\lambda_\alpha,\qquad\forall\,\vert\alpha\vert=d,\]
or in matrix form
\[\M_d(\blambda)\,\g\,=\,\frac{n+d}{n+2d}\,\blambda_d,\]
from which the desired result follows if $\M_d(\blambda)$ is non singular. But this follows from the fact that $G$ has nonempty interior.
\end{proof}
There are alternative ways for obtaining $\g$ from the moments $\blambda$. It suffices to apply (\ref{lem1-1}) for a family $\mathcal{F}$ of multi-indices 
$\alpha\in\N^n$ whose cardinal $\vert\mathcal{F}\vert$ matches the dimension ${n+d-1\choose d}$ of the vector $\g$.

For instance, with $n=2$ and $d=2$ ($g$ is a quadratic form with vector of coefficients $\g=(g_{20},g_{11},g_{02})$), $\g$ can also be obtained by:
\[\g\,=\,\left[\begin{array}{ccc}
\lambda_{20}&\lambda_{11}&\lambda_{02}\\
\lambda_{30}&\lambda_{21}&\lambda_{12}\\
\lambda_{21}&\lambda_{12}&\lambda_{30}\end{array}\right]^{-1}\left[\begin{array}{c}\frac{n}{n+2}\lambda_{00}\\
\frac{n+1}{n+3}\lambda_{10}\\ \frac{n+1}{n+3}\lambda_{01}\end{array}\right],\]
provided that the above inverse matrix exists.

Similarly, with $n=2$ and $d=4$ ($g$ is a quartic form with vector of coefficients $\g=(g_{40},g_{31},g_{22},g_{13},g_{04})$),
$\g$ can also be obtained by
\[\g\,=\,\left[\begin{array}{ccccc}
\lambda_{40}&\lambda_{31}&\lambda_{22}&\lambda_{13}&\lambda_{04}\\
\lambda_{50}&\lambda_{41}&\lambda_{32}&\lambda_{23}&\lambda_{14}\\
\lambda_{41}&\lambda_{32}&\lambda_{23}&\lambda_{14}&\lambda_{05}\\
\lambda_{60}&\lambda_{51}&\lambda_{42}&\lambda_{33}&\lambda_{24}\\
\lambda_{42}&\lambda_{33}&\lambda_{24}&\lambda_{15}&\lambda_{06}\\
\end{array}\right]^{-1}\left[\begin{array}{c}\frac{n}{n+4}\lambda_{00}\\
\frac{n+1}{n+5}\lambda_{10}\\ \frac{n+1}{n+5}\lambda_{01}\\ \frac{n+2}{n+6}\lambda_{20}\\ \frac{n+2}{n+6}\lambda_{02}
\end{array}\right],\]
provided that the above inverse matrix exists.
\end{document}